\documentclass[12pt,reqno]{amsart}
\usepackage{amscd,amssymb,amsmath,tikz}
\begin{document}

\newtheorem{thm}[equation]{Theorem}
\numberwithin{equation}{section}
\newtheorem{cor}[equation]{Corollary}
\newtheorem{expl}[equation]{Example}
\newtheorem{rmk}[equation]{Remark}
\newtheorem{conv}[equation]{Convention}
\newtheorem{claim}[equation]{Claim}
\newtheorem{lem}[equation]{Lemma}
\newtheorem{sublem}[equation]{Sublemma}
\newtheorem{conj}[equation]{Conjecture}
\newtheorem{defin}[equation]{Definition}
\newtheorem{diag}[equation]{Diagram}
\newtheorem{prop}[equation]{Proposition}
\newtheorem{notation}[equation]{Notation}
\newtheorem{tab}[equation]{Table}
\newtheorem{fig}[equation]{Figure}
\newcounter{bean}
\renewcommand{\theequation}{\thesection.\arabic{equation}}

\raggedbottom \voffset=-.7truein \hoffset=0truein \vsize=8truein
\hsize=6truein \textheight=8truein \textwidth=6truein
\baselineskip=18truept

\def\mapright#1{\ \smash{\mathop{\longrightarrow}\limits^{#1}}\ }
\def\mapleft#1{\smash{\mathop{\longleftarrow}\limits^{#1}}}
\def\mapup#1{\Big\uparrow\rlap{$\vcenter {\hbox {$#1$}}$}}
\def\mapdown#1{\Big\downarrow\rlap{$\vcenter {\hbox {$\ssize{#1}$}}$}}
\def\mapne#1{\nearrow\rlap{$\vcenter {\hbox {$#1$}}$}}
\def\mapse#1{\searrow\rlap{$\vcenter {\hbox {$\ssize{#1}$}}$}}
\def\mapr#1{\smash{\mathop{\rightarrow}\limits^{#1}}}
\def\ss{\smallskip}
\def\s{\sigma}
\def\l{\lambda}
\def\vp{v_1^{-1}\pi}
\def\at{{\widetilde\alpha}}

\def\sm{\wedge}
\def\la{\langle}
\def\ra{\rangle}
\def\ev{\text{ev}}
\def\od{\text{od}}
\def\on{\operatorname}
\def\ol#1{\overline{#1}{}}
\def\spin{\on{Spin}}
\def\cat{\on{cat}}
\def\lbar{\ell}
\def\qed{\quad\rule{8pt}{8pt}\bigskip}
\def\ssize{\scriptstyle}
\def\a{\alpha}
\def\bz{{\Bbb Z}}
\def\Rhat{\hat{R}}
\def\im{\on{im}}
\def\ct{\widetilde{C}}
\def\ext{\on{Ext}}
\def\sq{\on{Sq}}
\def\eps{\epsilon}
\def\ar#1{\stackrel {#1}{\rightarrow}}
\def\br{{\bold R}}
\def\bC{{\bold C}}
\def\bA{{\bold A}}
\def\bB{{\bold B}}
\def\bD{{\bold D}}
\def\bC{{\bold C}}
\def\bh{{\bold H}}
\def\bQ{{\bold Q}}
\def\bP{{\bold P}}
\def\bx{{\bold x}}
\def\bo{{\bold{bo}}}
\def\dh{\widehat{d}}
\def\si{\sigma}
\def\Vbar{{\overline V}}
\def\dbar{{\overline d}}
\def\wbar{{\overline w}}
\def\Sum{\sum}
\def\tfrac{\textstyle\frac}

\def\tb{\textstyle\binom}
\def\Si{\Sigma}
\def\w{\wedge}
\def\equ{\begin{equation}}
\def\b{\beta}
\def\G{\Gamma}
\def\L{\Lambda}
\def\g{\gamma}
\def\d{\delta}
\def\k{\kappa}
\def\psit{\widetilde{\Psi}}
\def\tht{\widetilde{\Theta}}
\def\psiu{{\underline{\Psi}}}
\def\thu{{\underline{\Theta}}}
\def\aee{A_{\text{ee}}}
\def\aeo{A_{\text{eo}}}
\def\aoo{A_{\text{oo}}}
\def\aoe{A_{\text{oe}}}
\def\vbar{{\overline v}}
\def\endeq{\end{equation}}
\def\sn{S^{2n+1}}
\def\zp{\bold Z_p}
\def\cR{{\mathcal R}}
\def\P{{\mathcal P}}
\def\cQ{{\mathcal Q}}
\def\cj{{\cal J}}
\def\zt{{\bold Z}_2}
\def\bs{{\bold s}}
\def\bof{{\bold f}}
\def\bq{{\bold Q}}
\def\be{{\bold e}}
\def\Hom{\on{Hom}}
\def\ker{\on{ker}}
\def\kot{\widetilde{KO}}
\def\coker{\on{coker}}
\def\da{\downarrow}
\def\colim{\operatornamewithlimits{colim}}
\def\zphat{\bz_2^\wedge}
\def\io{\iota}
\def\om{\omega}
\def\Prod{\prod}
\def\e{{\cal E}}
\def\zlt{\Z_{(2)}}
\def\exp{\on{exp}}
\def\abar{{\overline a}}
\def\xbar{{\overline x}}
\def\ybar{{\overline y}}
\def\zbar{{\overline z}}
\def\mbar{{\overline m}}
\def\nbar{{\overline n}}
\def\sbar{{\overline s}}
\def\kbar{{\overline k}}
\def\bbar{{\overline b}}
\def\et{{\widetilde E}}
\def\ni{\noindent}
\def\tsum{\textstyle \sum}
\def\coef{\on{coef}}
\def\den{\on{den}}
\def\lcm{\on{l.c.m.}}
\def\vi{v_1^{-1}}
\def\ot{\otimes}
\def\psibar{{\overline\psi}}
\def\thbar{{\overline\theta}}
\def\mhat{{\hat m}}
\def\exc{\on{exc}}
\def\ms{\medskip}
\def\ehat{{\hat e}}
\def\etao{{\eta_{\text{od}}}}
\def\etae{{\eta_{\text{ev}}}}
\def\dirlim{\operatornamewithlimits{dirlim}}
\def\gt{\widetilde{L}}
\def\lt{\widetilde{\lambda}}
\def\st{\widetilde{s}}
\def\ft{\widetilde{f}}
\def\sgd{\on{sgd}}
\def\lfl{\lfloor}
\def\rfl{\rfloor}
\def\ord{\on{ord}}
\def\gd{{\on{gd}}}
\def\rk{{{\on{rk}}_2}}
\def\nbar{{\overline{n}}}
\def\MC{\on{MC}}
\def\lg{{\on{lg}}}
\def\cH{\mathcal{H}}
\def\cS{\mathcal{S}}
\def\cP{\mathcal{P}}
\def\N{{\Bbb N}}
\def\Z{{\Bbb Z}}
\def\Q{{\Bbb Q}}
\def\R{{\Bbb R}}
\def\C{{\Bbb C}}

\def\mo{\on{mod}}
\def\xt{\times}
\def\notimm{\not\subseteq}
\def\Remark{\noindent{\it  Remark}}
\def\kut{\widetilde{KU}}

\def\*#1{\mathbf{#1}}
\def\0{$\*0$}
\def\1{$\*1$}
\def\22{$(\*2,\*2)$}
\def\33{$(\*3,\*3)$}
\def\ss{\smallskip}
\def\ssum{\sum\limits}
\def\dsum{\displaystyle\sum}
\def\la{\langle}
\def\ra{\rangle}
\def\on{\operatorname}
\def\proj{\on{proj}}
\def\od{\text{od}}
\def\ev{\text{ev}}
\def\o{\on{o}}
\def\U{\on{U}}
\def\lg{\on{lg}}
\def\a{\alpha}
\def\bz{{\Bbb Z}}
\def\eps{\varepsilon}
\def\bc{{\bold C}}
\def\bN{{\bold N}}
\def\bB{{\bold B}}
\def\bW{{\bold W}}
\def\nut{\widetilde{\nu}}
\def\tfrac{\textstyle\frac}
\def\b{\beta}
\def\G{\Gamma}
\def\g{\gamma}
\def\zt{{\Bbb Z}_2}
\def\zth{{\bold Z}_2^\wedge}
\def\bs{{\bold s}}
\def\bx{{\bold x}}
\def\bof{{\bold f}}
\def\bq{{\bold Q}}
\def\be{{\bold e}}
\def\lline{\rule{.6in}{.6pt}}
\def\xb{{\overline x}}
\def\xbar{{\overline x}}
\def\ybar{{\overline y}}
\def\zbar{{\overline z}}
\def\ebar{{\overline \be}}
\def\nbar{{\overline n}}
\def\ubar{{\overline u}}
\def\bbar{{\overline b}}
\def\et{{\widetilde e}}
\def\lf{\lfloor}
\def\rf{\rfloor}
\def\ni{\noindent}
\def\ms{\medskip}
\def\Dhat{{\widehat D}}
\def\what{{\widehat w}}
\def\Yhat{{\widehat Y}}
\def\abar{{\overline{a}}}
\def\minp{\min\nolimits'}
\def\sb{{$\ssize\bullet$}}
\def\mul{\on{mul}}
\def\N{{\Bbb N}}
\def\Z{{\Bbb Z}}
\def\S{\Sigma}
\def\Q{{\Bbb Q}}
\def\R{{\Bbb R}}
\def\C{{\Bbb C}}
\def\Xb{\overline{X}}
\def\eb{\overline{e}}
\def\notint{\cancel\cap}
\def\cS{\mathcal S}
\def\cR{\mathcal R}
\def\el{\ell}
\def\TC{\on{TC}}
\def\GC{\on{GC}}
\def\wgt{\on{wgt}}
\def\Ht{\widetilde{H}}
\def\wbar{\overline w}
\def\dstyle{\displaystyle}
\def\Sq{\on{sq}}
\def\Om{\Omega}
\def\ds{\dstyle}
\def\tz{tikzpicture}
\def\zcl{\on{zcl}}
\def\bd{\bold{d}}
\def\io{\iota}
\def\Vb#1{{\overline{V_{#1}}}}
\def\Ebar{\overline{E}}

\title
{The mod-2 connected $KU$-homology of the Eilenberg-MacLane space $K(\zt,2)$}
\author{Donald M. Davis}
\address{Department of Mathematics, Lehigh University\\Bethlehem, PA 18015, USA}
\email{dmd1@lehigh.edu}
\author{W. Stephen Wilson}
\address{Department of Mathematics, Johns Hopkins University\\Baltimore, MD 01220, USA}
\email{wwilson3@jhu.edu}
\date{June 24, 2021}

\keywords{K-theory, Eilenberg-MacLane space, Adams spectral sequence}
\thanks {2000 {\it Mathematics Subject Classification}: 55T15, 55P20, 55N15.}

\maketitle

\begin{abstract} We compute the mod-2 connected $KU$-homology of the Eilenberg-MacLane space $K(\zt,2)$, using a novel Adams spectral sequence analysis.
\end{abstract}

\section{Introduction}\label{intro}
Let $k(1)$ denote the connected spectrum for mod-2 complex $K$-theory, and $k(1)_*(-)$ the associated reduced homology theory. It is the connected version of the Morava $K$-theory $K(1)$. The coefficient ring is $k(1)_*=\zt[v]$, with $|v|=2$, and the mod-2 cohomology groups as a module over the mod-2 Steenrod algebra $A$ are given by $H^*(k(1))\approx A/A(Q_1)$, where $Q_1$ is the Milnor primitive of grading 3. (All cohomology groups have coefficients in $\zt$.) By a standard change-of-rings theorem, the Adams spectral sequence (ASS) converging to $k(1)_*(X)$ has $E_2=\ext_{E(Q_1)}(\Ht^*(X),\zt)$, where $E(Q_1)$ is the exterior algebra over $\zt$ with single generator $Q_1$.  Any $E(Q_1)$-module $\Ht^*X$ splits as $F(X)\oplus T(X)$, where $F(X)$ is a free $E(Q_1)$-module, and $T(X)$ a trivial $E(Q_1)$-module. The $E_2$-term of the ASS converging to $k(1)_*(X)$ consists of filtration-0 $\zt$'s annihilated by $v$, corresponding to an $E(Q_1)$-basis of $F(X)$, and $v$-towers $\{v^ix:\ i\ge0\}$ for $x$ in a $\zt$-basis of $T(X)$. See Figure \ref{fig1} for a depiction of two $v$-towers.

In this paper, we compute the ASS for $k(1)_*(K_2)$, where $K_2$ is the Eilenberg-MacLane space $K(\zt,2)$. One reason for doing this, aside from its intrinsic interest, is to develop methods that might be useful in computing integral complex and real connective $KU$- and $KO$-homology groups $ku_*(K(\zt,n))$ and $ko_*(K(\zt,n))$.  Partial calculations of these were made for small values of $n$ in \cite{DW} in order to detect Spin manifolds with nonzero dual Stiefel-Whitney classes.

We now develop some notation which will be used in stating our result. Let $P$ denote a polynomial algebra and $E$ an exterior algebra over $\zt$. For an algebra $B$, we let $\overline{B}=I(B)$ denote the kernel of the augmentation, and for an exterior algebra $E[S]$, we write $\Ebar[S]$ for $\overline{E[S]}$.
We will show in Section \ref{E2sec} that
\begin{equation}\label{E2}\ext_{E(Q_1)}(T(K_2),\zt)=P[v]\otimes \Ebar[x_9,\ x_{17},\ x_{2^e},e\ge2,\ x_{2^{e+2}+2},e\ge2],\end{equation}
with $|x_i|=i$ and $v\in\ext_{E(Q_1)}^{1,3}(\zt,\zt)$. The bulk of our proof will be in finding the pattern of differentials among these $v$-towers. In Figure \ref{fig1}, we illustrate our first differential, $d_2(x_9)=v^2x_4$, using the usual $(t-s,s)$ depiction of the ASS. This will leave in $k(1)_*(K_2)$ a truncated $v$-tower on $x_4$ of height 2; i.e., nonzero elements $x_4$ and $vx_4$.

\bigskip
\begin{minipage}{6in}
\begin{fig}\label{fig1}

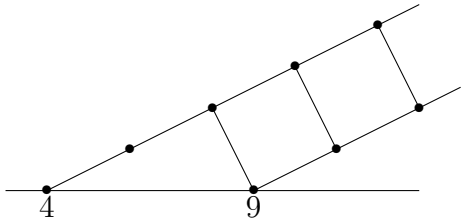
{\bf A differential in the ASS of $k(1)_*(K_2)$}

\begin{center}

\begin{\tz}[scale=.55]
\draw (4,0) -- (13,4.5);
\draw (3,0) -- (13,0);
\draw (9,0) -- (14,2.5);
\draw (9,0) -- (8,2);
\draw (11,1) -- (10,3);
\draw (13,2) -- (12,4);
\node at (4,-.4) {$4$};
\node at (9,-.4) {$9$};
\node at (4,0) {$\ssize\bullet$};
\node at (9,0) {$\ssize\bullet$};
\node at (6,1) {$\ssize\bullet$};
\node at (8,2) {$\ssize\bullet$};
\node at (10,3) {$\ssize\bullet$};
\node at (12,4) {$\ssize\bullet$};
\node at (11,1) {$\ssize\bullet$};
\node at (13,2) {$\ssize\bullet$};
\end{\tz}
\end{center}
\end{fig}
\end{minipage}

\bigskip
Let $y_e=x_{2^e}$, $z_e=x_{2^{e+2}+2}$, $p_2=x_9$, $p_3=x_{17}$, and for $e\ge2$, $p_{e+2}=p_ey_ez_e$. Let
$$\L_{k,\ell}=E[y_i,i\ge k,\ z_j,j\ge\ell].$$
Let
$$h(e)=\frac{2^e+(-1)^{e+1}}3+\biggl\lfloor\frac e2\biggr\rfloor \quad\text{and}\quad h'(e)=\frac{2^{e+1}+(-1)^e}3-\biggl\lfloor \frac e2\biggr\rfloor.$$
Note that $h'(1)=1$,  $h(2)=h'(2)=2$, $h(3)=h'(3)=4$, and for $e\ge3$, $h'(e)<h(e+1)<h'(e+1)$. For $e=4$, 5, 6, 7, and 8, $h(e)$ is 7, 13, 24, 46, 89, while $h'(e)$ is 9, 19, 40, 82, and 167. One easily establishes, by induction on $e$, that
\begin{equation}\label{ph}|p_e|=2h(e)+2^e+1.\end{equation}

Our main theorem is
\begin{thm}\label{main} In the $v$-tower portion of the ASS converging to $k(1)_*(K_2)$, whose $E_2$ term was given in (\ref{E2}), there are differentials $$d_{h(e)}(p_eM)=v^{h(e)}y_eM\quad\text{and}\quad d_{h'(e)}(z_eM')=v^{h'(e)}y_ep_eM'$$ for $e\ge2$, annihilating all $v$-towers, leaving just $v$-torsion of the sort described just before Figure \ref{fig1}. Here $M$ is any monomial in $E[p_{e+1}]\otimes\L_{e+1,e}$, and $M'$ is any monomial in $E[p_{e+1}]\otimes \L_{e+1,e+1}$. Then $k(1)_*(K_2)$ consists of the truncated $v$-towers just described plus a family of $\zt$'s annihilated by $v$, enumerated in Proposition \ref{z2s}.\end{thm}

In Table \ref{T4}, we list the first few generators, including their grading, the $v$ height, and the gradings of the first four exterior generators with which the $v$-tower is tensored.

\begin{table}[h]
\caption{First few generators}
\label{T4}

\begin{tabular}{cccc}
generator&grading&$v$-height&$M^{(')}\ldots$\\
\hline
$y_2$&$4$&$2$&$8,16,17,18$\\
$y_2p_2$&$13$&$2$&$8,16,17,32$\\
$y_3$&$8$&$4$&$16,31,32,34$\\
$y_3p_3$&$25$&$4$&$16,31,32,64$\\
$y_4$&$16$&$7$&$32,59,64,66$\\
$y_4p_4$&$47$&$9$&$32,59,64,128$\\
$y_5$&$32$&$13$&$64,113,128,130$\\
\end{tabular}
\end{table}

It was shown in \cite[Appendix]{JW}, building on \cite{RW2}, that, for the first periodic Morava $K$-theory $K(1)$, one has $K(1)_*(K_2)=0$. This is just the $v$-localization of our $k(1)_*(K_2)$, and so is equivalent to the fact that no infinite $v$-towers survive. Our proof does not use the fact that $K(1)_*(K_2)=0$, and thus gives a new proof of this, relying just on ASS manipulations. This also follows from \cite{AH}, where they show $KU^*(K_2) = 0$.

 Although it was \cite{DW} that motivated this paper, the known result that $K(1)_*(K_2) = 0$ and the small size of $H^*(K_2)$ inspire a search for the differentials in the Atiyah-Hirzebruch spectral sequence (AHSS) that annihilate everything.  The second author attempted to find these differentials periodically for decades.  The ASS differentials $d_r$ that we find in this paper are the same as the AHSS differentials $d_{2r+1}$ (the proof is the same) and solves the problem.

\section{Determining $E_2$}\label{E2sec}
In this section, we compute the $E_2$ term of the ASS converging to $k(1)_*(K_2)$. It is well-known (\cite{Ser}) that $H^*(K_2)$ is a polynomial algebra on elements $u_{2^i+1}$ for $i\ge0$ defined by
$$u_2=\io_2,\ u_3=\sq^1\io_2,\ u_5=\sq^2\sq^1\io_2,\ u_9=\sq^4\sq^2\sq^1\io_2,\ldots.$$
We have $Q_1u_2=u_5$, $Q_1u_3=u_3^2$, $Q_1u_5=0$, and for $i\ge1$,
\begin{eqnarray*}Q_1u_{2^{i+2}+1}&=&(\sq^3+\sq^2\sq^1)\sq^{2^{i+1}}\sq^{2^i}\cdots\sq^1\io_2\\
&=&\sq^{2^{i+1}+3}\sq^{2^i}\cdots\sq^1\io_2+\sq^{2^{i+1}+2}\sq^{2^i+1}(\sq^{2^i-1}\cdots\sq^1\io_2)\\
&=&0+(\sq^{2^i-1}\cdots\sq^1\io_2)^4=u_{2^i+1}^4.\end{eqnarray*}

Let $TP_n[x]=P[x]/x^n$. From the $Q_1$-action just described, we easily obtain that $H^*(K_2)=T(K_2)\otimes F'(K_2)$, where
\begin{equation}\label{T}T(K_2)=P[u_2^2]\otimes TP_4[u_9+u_3^3]\otimes TP_4[u_{17}+u_2u_5^3]\otimes\bigotimes_{i\ge5}E[u_{2^i+1}^2]\end{equation}
is a trivial $E(Q_1)$-submodule of $H^*(K_2)$, corresponding to the homology of $H^*K_2$ with respect to the differential defined by the $Q_1$-action, and
$$F'(K_2)=(E[u_2]\ot P[u_5])\ot(E[u_3]\ot P[u_3^2])\ot\bigotimes_{i\ge3}(E[u_{2^{i+2}+1}]\ot P[u_{2^i+1}^4])$$
satisfies that $\overline{F'(K_2)}$ is a free $E(Q_1)$-submodule. Then $F(K_2)=\overline{F'(K_2)}\ot T(K_2)$, and we obtain the following result for the filtration-0 $\zt$'s in $k(1)_*(K_2)$ annihilated by $v$, which we mostly ignore.
\begin{prop}\label{z2s} The Poincar\'e series for the filtration-$0$ $\zt$'s in $k(1)_*(K_2)$ annihilated by $v$ is $P_{\overline{F'}}\cdot P_T$, where
$$P_{\overline{F'}}=\frac{1+x^2}{1-x^5}\cdot\frac{1}{1-x^6}\cdot\prod_{i\ge5}\frac{1+x^{2^i+1}}{1-x^{2^i+4}}-\frac1{1+x^3}$$
and
$$P_T= \frac1{1-x^4}\cdot\frac{1-x^{36}}{1-x^9}\cdot\frac{1-x^{68}}{1-x^{17}}\cdot\prod_{i\ge5}(1+x^{2^{i+1}+2}).$$
\end{prop}
\begin{proof} The formula for $P_{\overline{F'}}$ is obtained from the Poincar\'e series for all of $\overline{F'}$ divided by $(1+x^3)$.\end{proof}

The product structure of $\ext^0_{E(Q_1)}(T(K_2),\zt)$ will be important in studying the differentials in the ASS. This is just the $\zt$-dual $(T(K_2))^*$. The $H$-space multiplication of $K_2$ makes $H^*K_2$ and $H_*K_2$ into dual Hopf algebras, and $H_*K_2$ is an exterior algebra, as it is dual to the primitively-generated polynomial algebra $H^*K_2$.
The action of $Q_1$ makes these into dual differential Hopf algebras, and by \cite[Lemma 4.5]{Br} their homologies with respect to this differential are also dual Hopf algebras. Thus $(T(K_2))^*$ is a Hopf algebra which is a subquotient of the exterior algebra $H_*K_2$. Since exterior algebras over $\zt$ are characterized by the property that squares of nonunits are 0, a subquotient of an exterior algebra is an exterior algebra. Hence $\ext^0_{E(Q_1)}(T(K_2),\zt)$ is an exterior algebra, and, counting generators in (\ref{T}), we obtain (\ref{E2}).

\section{Differentials}
In this section we establish the differentials in the ASS for $k(1)_*(K_2)$, which were stated in Theorem \ref{main}. Our proof will use the following lemmas.
\begin{lem}\label{LDL} In an ASS in which $E_r=P[v]\ot E[x,y]\ot B$, for some algebra $B$, suppose $d_r(x)=v^ry$. Then $d_r(y)=0$, $d_r(xy)=0$, and, after the differential, the $v$-tower part remaining is $P[v]\ot E[xy]\ot B$. In $E_{r+1}$, $xy$ is not a product class.
\end{lem}
\begin{proof} If $d_r(y)=v^rz\ne0$, we would have $d_r(d_r(y))=v^{2r}z\ne0$, contradicting $d_r^2=0$. Since $d_r$ is a derivation, $d_r(xy)=v^ry^2$, but $y^2=0$. The statement about what remains after the differential is true even if there are additional $d_r$-differentials in $B$, since we can consider the differentials as being filtered.\end{proof}
\begin{lem}\label{CP} In the ASS whose $E_2$-term was given in (\ref{E2}), $x_{2^e}$ is an infinite cycle, and $v^ix_{2^e}$ must be hit by a differential for some $i\le 2^{e-1}$.
\end{lem}
\begin{proof} Let $f:CP^\infty\to CP^\infty$ denote the $H$-space squaring map, and $g:CP^\infty\to K_2$ correspond to the nonzero element of $H^2(CP^\infty;\zt)$. The composite $g\circ f$ is trivial, and so $g_*:k(1)_*(CP^\infty)\to k(1)_*(K_2)$ sends all elements in $\im(k(1)_*(CP^\infty)\mapright{f_*}k(1)_*(CP^\infty))$ to 0. Let $\b_i\in k(1)_{2i}(CP^\infty)$ be dual to $y^i$, where $y$ generates $k(1)^*(CP^\infty)$. By \cite[Theorem 3.4]{RW}, $f_*(\b_{2i})=v^i\b_i$. Thus $v^ig_*\b_i=g_*(v^i\b_i)=0$. Since $g_*\b_i$ is dual to $u^i$, and $x_{2^e}$ is dual to $u^{2^{e-1}}$, we obtain that  $g_*(\b_{2^{e-1}})=x_{2^e}$, hence $v^{2^{e-1}}x_{2^e}=0$.
\end{proof}

The remainder of this note is devoted to the proof of Theorem \ref{main}.
The novelty of the proof is using higher differentials to deduce
lower differentials.
We need to get some terminology straight before we proceed.
We sometimes write $Vx$ to mean $v^i x$ for
some $i$.
We write $E_r$ for the pages of the ASS,
ignoring $v$-torsion.
For example, by (\ref{E2}),
we start with $E_2 = P[v]\otimes E[p_2,p_3]\otimes \Lambda_{2,2}$, which
ignores all of the elements on the zero line killed by multiplication by $v^1$.
When we have $d_r(y) = v^r x$, both the $y$ and $x$ are not in $E_{r+1}$.
Our preference is to think in terms of reduced groups, ignoring the $v$-tower
on $1$, so that our computations eventually kill all $v$-towers, but it
is just as well to leave the $v$-tower on $1$ there.

The following series of statements  will be proved by induction. They imply Theorem \ref{main}.

$$
\begin{array}{lcc}
E_{h'(e-1)+1} = P[v]\otimes E[p_e,p_{e+1}] \otimes \Lambda_{e,e} &	& (1,e) \\
&	& \\
d_{h(e)}(p_e) = v^{h(e)} y_e &\qquad \qquad	& (2,e) \\
&	& \\
E_{h(e)+1} = P[v]\otimes E[y_e p_e,p_{e+1}] \otimes \Lambda_{e+1,e}\text{ for }e\ge4 &	& (3,e) \\
&	& \\
d_{h'(e)}(z_e) = v^{h'(e)} y_e p_e &	& (4,e) \\
\end{array}
$$

Before we move on to the proof, we provide an overview of how it will go.

The only differentials
possible go from $v$-towers to $v$-towers.
Once $v$-torsion is created, its generator
cannot be a differential source and its truncated tower is too short to
be the target of a differential.

Our inductive proof loosely deals with some of the lowest degree elements
remaining at each stage.  Two differentials of different degree are forced
on us each time, namely $d_{h(e+2)}(p_{e+2}) = v^{h(e+2)}y_{e+2}$, $(2,e+2)$, and
$d_{h'(e)} = v^{h'(e)} = v^{h'(e)} y_e p_e$, $(4,e)$.  Since these
{\em must} happen,
the involved elements cannot be sources or targets of lower degree
differentials.  This is important, because, in principle, some differential
could be happening in some very high degree.  If this happened, the
$E_r$ we describe would be incorrect and we would need some subquotient
of our description instead.  However, as our induction proceeds and
we eventually get to that high degree, our results are forced and we
see that this could not have happened.  Because this possibility does
not happen and our proof shows it, we will ignore it and write the
$E_r$ as it really is without the fuss of the possibility that does not happen.

The derivation property of Adams differentials implies that the
differentials respect products.  The only minor concern for that comes early
with $e=2$ and $3$, due to the fact that $h(2)=h'(2)$ and $h(3)=h'(3)$, where $d_{h(e)}(p_e z_e) = v^{h(e)}(y_e z_e + p_e y_e p_e)
= y_e z_e$, since $p_e^2 = 0$.

We propose to prove our result using induction from the following statements:

$$
\begin{array}{lcl}
(1,e) \text{ and } (2,e) & \Rightarrow & (3,e) \\
&&\\
(3,e) \text{ and } (2,e+1) &\Rightarrow  &(2,e+2) \text{ and } (4,e) \\
&&\\
(3,e) \text{ and } (4,e) &\Rightarrow & (1,e+1)  \\
\end{array}
$$

We begin knowing statement $(1,2)$, which is just our $E_2$.
By Lemma \ref{CP}, $Vy_2$ must be hit by
a differential from an odd-dimensional class (abbr.~``hit by an odd class'')
of grading $\le 9$.  The only such class is $p_2$, so we obtain $d_2(p_2) = v^2y_2$,
which is statement $(2,2)$.  Using Lemma 3.1, we deduce a subquotient version of $(3,2)$. Note that $(3,2)$ is not valid because $h'(2)=h(2)$, so there is a second $d_2$-differential. A similar situation happens when $e=3$. Since $h'(3)=h(3)$, $(3,3)$ is only true as a subquotient because $(4,3)$ gives a second $d_4$-differential.  By Lemma \ref{CP},
$Vy_3$ must be hit by an odd class of grading $\le 17$.  By Lemma \ref{LDL}, it cannot
be $d_2(y_2p_2)$.  Therefore it must be $d_4(p_3) = v^4 y_3$.  This gives $(2,3)$.

That started  our induction.

The first and third induction statements follow directly from Lemma \ref{LDL}, together with our earlier remarks that there cannot be extraneous differentials.
The second line is the heart of the proof.
By $(3,e)$ we are working with generators in $E[y_e p_e, p_{e+1}]\otimes
\Lambda_{e+1,e}$.
If we applied $(2,e+1)$ to this, Lemma \ref{LDL} would give us $v$-towers on
$E[y_e p_e, y_{e+1} p_{e+1}]\otimes \Lambda_{e+2,e}$.
By Lemma \ref{CP}
and the derivation property, $Vy_{e+2}$ must be hit by an odd generator $x$
with
$2^{e+2} + 5 \le |x| \le 2^{e+3}+1$.
Only one of the two odd degree generators, $y_{e+1}p_{e+1}$, is in this range, but
by Lemma \ref{LDL} and $(2,e+1)$, $d_{h(e+1)}(y_{e+1}p_{e+1}) = 0$, and this is
exactly the differential that would be required to hit $Vy_{e+2}$.
We
still need an odd generator in the appropriate range.  The only way to get
this is using Lemma \ref{LDL} and the differential $(4,e)$.  This creates a generator
$y_e p_e \cdot z_e = p_{e+2}$, enabling the differential $(2,e+2)$.

\def\line{\rule{.6in}{.6pt}}

\end{document}